\documentclass[12pt]{amsart}

\newtheorem{theorem}{Theorem}[section]

\newtheorem{lemma}[theorem]{Lemma}

\newtheorem{remark}{Remark}[section]
\newtheorem{mainremark}{Remark}
\newtheorem{definition}{Definition}[section]
\newtheorem{maintheorem}{Theorem}
\newtheorem{maincorollary}{Corollary}
\newtheorem{claim}{Claim}[section]
\newtheorem{conjecture}{Conjecture}
\newcommand{\blanksquare}{\,\,\,$\sqcup\!\!\!\!\sqcap$}

\newcounter{example}
{\par}

\usepackage{amscd}
\usepackage{threeparttable}
\usepackage{amssymb}
\usepackage{amsfonts}
\usepackage{epsfig}
\usepackage{graphicx}
\usepackage{amsmath}
\usepackage{color}
\usepackage[all, 2cell]{xy} \UseAllTwocells \SilentMatrices
\usepackage{indentfirst}

\newcommand{\ti}{\;\;\makebox[0pt]{$\top$}\makebox[0pt]{$\cap$}\;\;}

\begin{document}

\title[Stability properties of divergence-free vector fields]{Stability properties of divergence-free vector fields}

\author[C\'elia Ferreira]{C\'elia Ferreira}
\address{Departamento de Matem\'atica, Universidade do Porto, 
Rua do Campo Alegre, 687, 
4169-007 Porto, Portugal}
\email{celiam@fc.up.pt}

\begin{abstract}
A divergence-free vector field satisfies the star property if any divergence-free vector field in some $C^1$-neighborhood has all singularities and all closed orbits hyperbolic. 

In this paper we prove that any divergence-free vector field defined on a Riemannian manifold and satisfying the star property is Anosov. 
It is also shown that a $C^1$-structurally stable divergence-free vector field can be approximated by an Anosov divergence-free vector field.
Moreover, we prove that any divergence-free vector field can be $C^1$-approximated by an Anosov divergence-free vector field, or else by a divergence-free vector field exhibiting a heterodimensional cycle.
\end{abstract}

\maketitle



\begin{section}{Introduction and statement of the results}\label{resultstat}

Let $M$, sometimes denoted by $M^n$, be an $n$-dimensional, $n\geq 2$, closed, connected and smooth Riemannian manifold, endowed with a volume form, which has associated a measure $\mu$, called the Lebesgue measure. 
Let $\mathfrak{X}^r(M)$ be the set of vector fields and let $\mathfrak{X}_{\mu}^r(M)$ be the set of divergence-free vector fields, both defined on $M$ and endowed with the $C^r$ Whitney topology, $r\geq 1$. 
We emphasize that in this paper we are restricted to the $C^1$ topology, since our proofs use several technical results which are just proved for this topology (see results in Section~\ref{auxresults}).

Take $X\in\mathfrak{X}^r(M)$. We denote by $Per(X)$ the union of the closed orbits of $X$ and by $Sing(X)$ the union of the singularities of $X$. Singularities and closed orbits are called \textit{critical elements}, denoted by $Crit(X)$.
If $p\notin Sing(X)$ then $p$ is called a \textit{regular} point and $M$ is said to be regular if $Sing(X)=\emptyset$.

Take $x\in M$ a regular point for $X\in\mathfrak{X}^1(M)$ and let $N_x:=X(x)^{\bot}\subset T_xM$ denote the ($\dim(M)-1$)-di\-men\-sion\-al normal bundle of $X$ at $x$. Since, in general, $N_x$ is not $DX^t_x$-invariant, we define the \textit{linear Poincar\'{e} flow} $$P^t_X(x):=\Pi_{X^t(x)}\circ DX^t_x,$$ where $\Pi_{X^t(x)}: T_{X^t(x)}M\rightarrow N_{X^t(x)}$ is the canonical orthogonal projection.
Recently, Li, Gan and Wen generalized the notion of the linear Poincar\'{e} flow, in order to include singularities in their study (see \cite{LGW}).

Now, we state some basic definitions. 

\begin{definition}
Let $X\in\mathfrak{X}^1(M)$. An $X^t$-invariant, compact and regular set $\Lambda\subset M$ is called \emph{hyperbolic} if $N_{\Lambda}$ has a $P_X^t$-invariant splitting $N_{\Lambda}^s\oplus N_{\Lambda}^u$ such that there is $\ell>0$ satisfying 
\begin{center}
$\|P_X^{\ell}(x)|_{N^s_x}\|\leq \dfrac{1}{2}$ and $\|P_X^{-{\ell}}(X^{\ell}(x))|_{N^u_{X^{\ell}(x)}}\|\leq\dfrac{1}{2}$, for any $x\in\Lambda$.
\end{center}
\end{definition}

A vector field $X$ is said to be \textit{Anosov} if the manifold $M$ is hyperbolic. Let $\mathcal{A}_{\mu}(M)$ denote the $C^1$-open set of divergence-free Anosov vector fields defined on $M$. 
In the divergence-free context, a hyperbolic critical point $p$ must be of \textit{saddle} type, having both $N_{p}^s$ and $N_{p}^u$ with dimension between $1$ and $n-2$.
A vector field $X$ is called \textit{isolated in the boundary of $\mathcal{A}_{\mu}(M)$} if $X$ is not Anosov and, given a small neighborhood $\mathcal{U}$ of $X$, any vector field $Y\in\mathcal{U}\backslash X$ is Anosov.

Now, we state the definition of \emph{dominated splitting}, a more relaxed notion of splitting.

\begin{definition}
Take $X\in\mathfrak{X}^1(M)$ and $\Lambda\subset M$ a compact, $X^t$-invariant and regular set. A $P_X^t$-invariant splitting $N_{\Lambda}=N_{\Lambda}^1\oplus N_{\Lambda}^2$ is said a \emph{dominated splitting} if there is $\ell>0$ such that
\begin{center}
${\|P_X^{\ell}(x)|_{{{N}}_x^1}\|}\cdot{\|P_X^{-\ell}(X^{\ell}(x))|_{{{N}}_{X^{\ell}(x)}^2}\|}\leq\dfrac{1}{2}, \: \: \forall \: x\in \Lambda.$
\end{center}
\end{definition}


One of the most important conjectures in the field of dynamical systems, posed by Palis and Smale in 1970, is to know if a $C^r$-structural stable system satisfies the Axiom A and the strong transversality conditions. This is the so called \textit{structural stability conjecture}. Notice that, by the Poincar\'{e} recurrence theorem, in the conservative setting we conclude that the non-wandering set coincides with the whole manifold. So, a conservative diffeomorphism (or a divergence-free vector field) which is Axiom A is actually Anosov. 

The study of the previous conjecture motivated Ma\~n\'e, in the early 1980's, to define the set $\mathcal{F}^1(M)$ of dissipative diffeomorphisms having a $C^1$-neighborhood $\mathcal{U}$ such that every diffeomorphism inside $\mathcal{U}$ has all closed orbits of hyperbolic type.
We call $f\in\mathcal{F}^1(M)$ a \textit{star diffeomorphism}.

It is known that $\Omega$-stable diffeomorphisms belong to $\mathcal{F}^1(M)$ (see \cite{F}) and that if $f\in \mathcal{F}^1(M)$ then $\Omega(f)=\overline{Per(f)}$ (see \cite{M1}). Thus, the structural stability conjecture is contained in the following one

\begin{conjecture}
Does a star system have its non-wandering set hyperbolic?
\end{conjecture}

On \cite{M}, Ma\~n\'e proved the previous conjecture for surfaces: every dissipative diffeomorphism of $\mathcal{F}^1(M^2)$ satisfies the Axiom A and the no-cycle condition. Later, Hayashi extended this result for higher dimensions (see \cite{H}).  

In 1988, Ma\~n\'e presented a proof of the stability conjecture for $C^1$-diffeomorphisms (see \cite{Ma}).

For the continuous-time case, a star vector field is defined as follows.

\begin{definition}\label{starflowdef}
A vector field $X\in\mathfrak{X}^1(M)$ is a \textbf{\textit{star vector field}} if there exists a $C^1$-neighborhood $\mathcal{U}$ of $X$ in  $\mathfrak{X}^1(M)$ such that if $Y\in\mathcal{U}$ then every point in $Crit(Y)$ is hyperbolic. 
Moreover, a vector field $X\in\mathfrak{X}_{\mu}^1(M)$ is a \textbf{\textit{divergence-free star vector field}} if there exists a $C^1$-neighborhood $\mathcal{U}$ of $X$ in $\mathfrak{X}_{\mu}^1(M)$ such that if $Y\in\mathcal{U}$ then every point in $Crit(Y)$ is hyperbolic. 
The set of star vector fields is denoted by $\mathcal{G}^1(M)$ and the set of divergence-free star vector fields is denoted by $\mathcal{G}_{\mu}^1(M)$.
\end{definition}

Note that $\mathcal{G}^1(M)$ and $\mathcal{G}_{\mu}^1(M)$ are $C^1$-open in $\mathfrak{X}^1(M)$ and $\mathfrak{X}^1_{\mu}(M)$, respectively.

Once that the previous definition concerns critical elements only and the hyperbolicity put on critical elements is merely orbit-wise, the star-condition looks, a priori, quite weak. 
However, as we will see in Theorem \ref{BRGA} and Theorem \ref{mainth}, for the divergence-free setting it is not.

A star vector field may fail to have hyperbolic non-wandering set, as the famous Lorenz attractor shows (see \cite{G}), or may fail to have the critical elements dense in the non-wandering set (see \cite{Di}) or, even with Axiom A satisfied, still may fail to satisfy the no-cycle condition (see \cite{LW}). 
However, for star vector fields, all these counterexamples exhibit singularities. 
So, recently, Gan and Wen (see \cite{GW}) proved the following remarkable result about dissipative star vector fields defined on an $n$-dimensional manifold, where $n\geq3$:

\begin{theorem}
If $X\in\mathcal{G}^1(M^n)$ and $Sing(X)=\emptyset$ then $X$ is Axiom A without cycles. 
\end{theorem}

On \cite{GW}, it is also shown that a dissipative star vector field exhibit no heterodimensional cycles.

However, in the divergence-free setting it is possible to prove that a star vector field does not exhibit singularities. So, generalizing Gan and Wen's result, Bessa and Rocha (see \cite{BR2}) recently proved an analogous result on volume preserving vector fields, using techniques on conservative dynamics.

\begin{theorem}\label{BRGA}
If $X\in\mathcal{G}_{\mu}^1(M^3)$ then $Sing(X)=\emptyset$ and $X$ is Anosov.
\end{theorem}

From this result, note that we have $\mathcal{G}^1_{\mu}(M)=\mathcal{A}_{\mu}(M^3)$. 
However, this result cannot be trivially extended to higher dimensions because its proof, in dimension $3$, assumes that the normal bundle is splitted in two one-dimensional subbundles. 
So, using volume-preserving arguments, the authors were able to prove the existence of a dominated splitting for the linear Poincar\'{e} flow and then the hyperbolicity. 

On \cite{BFR}, the authors prove Conjecture 1 for Hamiltonian vector fields, defined on a three-dimensional, compact and with no singularities connected component of an energy level of a four-dimensional symplectic manifold.

In higher dimensions, the subbundles of the normal bundle may have dimension strictly larger than one, meaning that a vector field with a dominated splitting structure is not necessarily hyperbolic. 
In this paper, we prove the higher-dimensional version of Theorem \ref{BRGA}.

\begin{maintheorem}\label{mainth}
If $X\in\mathcal{G}^1_{\mu}(M^n)$ then $Sing(X)=\emptyset$ and $X$ is Anosov, $n\geq4$.
\end{maintheorem}

Notice that the converse of Theorem \ref{mainth} is trivially true due to the openness of Anosov vector fields set.
So, Theorem \ref{mainth} implies that 
$$\mathcal{G}^1(M^n)\cap \mathfrak{X}^1_{\mu}(M^n)=\mathcal{G}^1_{\mu}(M^n)=\mathcal{A}_{\mu}(M^n), \; n\geq 4.$$

The next result is a consequence of Theorem \ref{mainth}.

\begin{maincorollary}\label{maincor}
The boundary of the set $\mathcal{A}_{\mu}(M^n)$ has no isolated points, $n\geq 4$.
\end{maincorollary}

A vector field $X\in\mathfrak{X}^1(M)$ is said to be \textit{Kupka-Smale} if every element of $Crit(X)$ is hyperbolic and its invariant manifolds intersect transversely.
Considering $M$ a manifold with dimension greater than $3$, we have that the set of Kupka-Smale vector fields $\mathcal{KS}^1(M)$ is a $C^1$-residual subset of $\mathfrak{X}^1(M)$ (see \cite{S}). 
In \cite{R0}, it is shown that this property is also true for a residual subset of all divergence-free vector fields, meaning that the set of Kupka-Smale divergence-free vector fields $\mathcal{KS}^1_{\mu}(M)$ is a $C^1$-residual subset of $\mathfrak{X}_{\mu}^1(M)$.

\begin{mainremark}
From Theorem \ref{BRGA} and Theorem \ref{mainth}, it is straightforward to conclude that if $X\in\mathfrak{X}_{\mu}^1(M^n)$ is in the interior of $\mathcal{KS}^1_{\mu}(M^n)$ then $X\in\mathcal{A}_{\mu}(M^n)$, $n\geq3$. This is an immediate proof for divergence-free vector fields of the result shown by Toyoshiba in \cite{To} for vector fields.
\end{mainremark}

A vector field $X$ is $C^1$-\textit{structurally stable} if there exists a $C^1$-neigh\-bor\-hood $\mathcal{U}$ of $X$ in $\mathfrak{X}^1(M)$ such that every $Y\in\mathcal{U}$ is topologically conjugated to $X$ (see for instance \cite{PM}). The notion of structural stability was first introduced in the mid 1930's by Andronov and Pontrjagin (see \cite{AP}).

We point out that, after the proof of the $C^1$-structural stability conjecture for diffeomorphisms, Gan proved this conjecture for dissipative $C^1$-flows (see \cite{Ga}) and Bessa and Rocha presented a proof on \cite{BR2} for the $C^1$-divergence-free context, but considering a three-dimensional manifold. 
In this paper, we generalize this last result to higher dimensions.

\begin{maintheorem}\label{structh}
If $X\in\mathfrak{X}^1_{\mu}(M^n)$ is $C^1$-structurally stable then it can be $C^1$-approximated by $Y\in\mathcal{A}_{\mu}(M^n)$, $n\geq4$.
\end{maintheorem}

Nevertheless, the $C^r$-structural stability conjecture remains wide open for higher topologies ($r\geq2$). This is explained, in particular, because many of the $C^1$-perturbation arguments, as the closing lemma, the connecting lemma and the Franks lemma, are either unknown or they are false in higher topologies
(see further details in \cite{Gu,Pu,PS2}).

At the second half of the 1960's, it was already clear that uniform hyperbolicity could not be presented for every system of a dense subset in the universe of all dynamics. So, it triggered the start of the search of a answer to the question: Is it possible to look for a general scenario for dynamics? 
This search draw the attention to homoclinic orbits, that is, orbits that in the past and in the future converge to the same periodic orbit, which has been first considered by Poincar\'{e}, almost a century before. 
The creation or destruction of such orbits is, roughly speaking, what its meant by homoclinic bifurcations (see \cite{PT}). 
Based on these and other subsequent developments, Palis formulated, in the 1990's, the following conjecture (see \cite{PT,P}):

\begin{conjecture}
The diffeomorphisms exhibiting a homoclinic bifurcation are $C^r$-dense in the complement of the closure of the hyperbolic ones, $r\geq1$.
\end{conjecture}

Pujals and Sambarino (see \cite{PS}) provided a proof of this conjecture in the case of diffeomorphisms defined on a compact surface in the $C^1$ topology. Recently, Bessa and Rocha proved this conjecture for volume-preserving diffeomorphisms on \cite{BR3}. 
The authors show that a volume-preserving diffeomorphism can be $C^1$-approximated by an Anosov volume-preserving map, or else by a volume-preserving diffeomorphism displaying a heterodimensional cycle. 
The authors also show a similar result for symplectomorphisms.

On \cite{AH}, Arroyo and Hertz proved an analogous statement of the previous conjecture in the context of $C^1$-vector fields defined on a three-dimensional, compact manifold. In this context, besides \textit{homoclinic tangencies}, the \textit{singular cycles} are another homoclinic phenomenon that must be considered:  

\begin{theorem}
Any vector field $X\in \mathfrak{X}^1(M^3)$ can be approximated by another one $Y\in \mathfrak{X}^1(M^3)$ showing one of the following phenomena:
\begin{enumerate}
	\item Uniform hyperbolicity with the no-cycles condition;
	\item A homoclinic tangency;
	\item A singular cycle.
\end{enumerate}
\end{theorem}

On the conservative setting, Bessa and Rocha, considering a three-dimensional manifold $M$, proved the next result on \cite{BR1}:

\begin{theorem}
Any vector field $X\in\mathfrak{X}^1_{\mu}(M^3)$ can be $C^1$-approximated by another one $Y\in\mathfrak{X}^1_{\mu}(M^3)$ which is Anosov or else has a homoclinic tangency.
\end{theorem}

On that paper, the authors left open the following question: can any $X\in\mathfrak{X}^1_{\mu}(M^n)$ be $C^1$-approximated by a divergence-free vector field exhibiting some form of hyperbolicity in $M^n$, or by one exhibiting homoclinic tangencies or else by one having a heterodimensional cycle (see Definition \ref{heterocycle} bellow), for $n\geq 4$? 
In the present paper we also answer to this question.

\begin{maintheorem}\label{mainth2}
If $X\in\mathfrak{X}^1_{\mu}(M^n)$ then $X$ can be $C^1$-approximated by an Anosov divergence-free vector field, or else by a divergence-free vector field exhibiting a heterodimensional cycle, $n\geq 4$.
\end{maintheorem}

This result rules out the $C^1$-approximation by a vector field exhibiting a homoclinic tangency in the higher-dimensional divergence-free setting. 

This article is organized in five additional sections. 
In Section \ref{auxresults}, we have compiled some definitions and auxiliary results, that will be used to prove the main theorems. 
Section \ref{th1} presents the proof of Theorem~\ref{mainth}, which uses the results proved in Section \ref{auxlemmas}, and also contains the proof of Corollary \ref{maincor}. 
The proof of Theorem~\ref{structh} is provided in Section \ref{strstasec}. 
Finally, in Section~\ref{th2} we prove some auxiliary results that, jointly with Theorem~\ref{mainth}, allow us to easily conclude the proof of Theorem \ref{mainth2}.
\end{section}


\begin{section}{Definitions and auxiliary results}\label{auxresults}

In this section, we state some definitions and present some results that will be used in the proofs.

Let $\mathcal{O}_X(p)$ denote the orbit of $p\in Crit(X)$ and $\pi_X(p)$ its period. By period, we mean the least period. 
If $p$ is a singularity of $X$, we set $\pi_X(p)=0$ and $\mathcal{O}_X(p)=p$ and if $\mathcal{O}_X(p)$ is a hyperbolic set, its \textit{stable} and \textit{unstable} manifolds are defined as
$$W_X^s(\mathcal{O}_X(p))=\{q\in M: dist (X^t(q),\mathcal{O}_X(p))\rightarrow 0, t\rightarrow +\infty\} \ \ \text{and}$$
$$W_X^u(\mathcal{O}_X(p))=\{q\in M: dist (X^{-t}(q),\mathcal{O}_X(p))\rightarrow 0, t\rightarrow +\infty\}.$$ 

We observe that both $W_X^s(\mathcal{O}_X(p))$ and $W_X^u(\mathcal{O}_X(p))$ do not depend on $q\in \mathcal{O}_X(p)$. 
Therefore, we can write $W_X^s(\mathcal{O}_X(p))=W_X^s(q)$ and $W_X^u(\mathcal{O}_X(p))=W_X^u(q)$, for some $q\in \mathcal{O}_X(p)$.
These manifolds are respectively tangent to the subspaces $E_q^s\oplus \mathbb{R}X(q)$ and $\mathbb{R}X(q)\oplus E^u_q$ of $T_qM$, $q\in \mathcal{O}_X(p)$. Observe that $$\dim(W_X^s(\mathcal{O}_X(p)))+\dim(W_X^u(\mathcal{O}_X(p)))=\dim(M)+1.$$

Take $p\in Crit(X)$ a hyperbolic saddle for a vector field $X\in\mathfrak{X}^1(M)$. 
The \textit{index of $p$} is defined as the dimension of the unstable bundle $W_X^u(p)$ and will be denoted by $ind(p)$. 
Now, we state the notion of heterodimensional cycle.

\begin{definition}\label{heterocycle}
Take $X\in\mathfrak{X}^1(M)$ and let $p,q$ be two distinct hyperbolic critical points of saddle type such that $ind(p)<ind(q)$. A vector field $X$ exhibits a \textbf{\textit{heterodimensional cycle}} associated to $p$ and $q$ if the invariant manifolds of $p$ and $q$ intersect cyclically, that is $$W_X^s(p)\ti W_X^u(q)\neq \emptyset \text{ and } W_X^u(p)\cap W_X^s(q)\neq \emptyset,$$ where $\ti$ means that the intersection is transversal. 
\end{definition}

This definition can be trivially extended to a finite number of hyperbolic saddles.

\begin{remark}
The condition $ind(p)<ind(q)$, stated in the previous definition, ensures that the connection $W_X^s(p)\ti W_X^u(q)$ is \textit{$C^1$-persistent} and that the connection $W_X^u(p)\cap W_X^s(q)$ is \textit{not $C^1$-persistent}. 
\end{remark}

Let $\mathcal{HC}^1(M)$ and $\mathcal{HC}^1_{\mu}(M)$ denote the subsets of $\mathfrak{X}^1(M)$ and $\mathfrak{X}^1_{\mu}(M)$, respectively, whose elements exhibit heterodimensional cycles. A heterodimensional cycle is said to be \textit{periodic} if it is composed just by closed orbits, \textit{singular} if it is composed just by singularities and \textit{mixed} if it contains at least one singularity and one closed orbit.

A vector field $X\in\mathfrak{X}^1(M)$ is said to be \textit{far from heterodimensional cycles}, say $X\in \mathcal{FC}^1(M)$, if there exists a $C^1$-neighborhood $\mathcal {U}$ of $X$ on $\mathfrak{X}^1(M)$ such that every $Y\in\mathcal{U}$ does not exhibit heterodimensional cycles. 
Moreover, if $X$ and $\mathcal {U}$ are taken in $\mathfrak{X}^1_{\mu}(M)$ and $X$ satisfy the previous property, we say that $X\in\mathcal{FC}_{\mu}^1(M)$.

\begin{remark} We point out that:
\begin{itemize}
\item Heterodimensional cycles do not exist if $\dim(M)<3$ because, in this case, we cannot find hyperbolic critical points of saddle-type and with different indices.
\item If $\dim(M)=3$, $M$ does not support periodic heterodimensional cycles since, in this case, the stable and the unstable manifolds of any closed orbit are both two-dimensional. 
However, it is possible to find singular heterodimensional cycles and also mixed heterodimensional cycles, where a link connecting two closed orbits is not allowed. 
Mixed heterodimensional cycles just appear in the case that the singularities have index $2$, since the index of every closed orbit is $1$.
\end{itemize}
\end{remark}

The first auxiliary result stated is due to Zuppa (see \cite{Z}) and allows us to $C^1$-ap\-prox\-i\-mate any divergence-free vector field by a $C^{\infty}$ one keeping the divergence-free property.

\begin{theorem}\label{zuppa}
The set of $C^{\infty}$ divergence-free vector fields is $C^1$-dense in $\mathfrak{X}^1_{\mu}(M)$.
\end{theorem}

The next result is a Pasting Lemma (see \cite{AM}) and it allows us to realize $C^1$-local perturbations in the divergence-free setting.

\begin{theorem}\label{pasting}
Given $\epsilon>0$ there exists $\delta>0$ such that if $X\in \mathfrak{X}^{1}_{\mu}(M)$, $K\subset M$ is a compact set and $Y\in \mathfrak{X}^{\infty}_{\mu}(M)$ is $\delta$-$C^1$-close to $X$ in a small neighborhood $U\supset K$, then there exist $Z\in \mathfrak{X}^{\infty}_{\mu}(M)$ and open sets $V$ and $W$, such that $K\subset V\subset U \subset W$, satisfying the properties:
\begin{itemize}
	\item $Z|_V=Y$;
	\item $Z|_{int(W^c)}=X$;
	\item $Z$ is $\epsilon$-$C^1$-close to $X$.
\end{itemize}
\end{theorem}

The following result is a version of Franks' lemma for divergence-free vector fields (see \cite{BR} for more details). 
Under some conditions, it allows us to realize a perturbation on the linear Poincar\'{e} flow by a vector field which is $C^1$-close to the original one.

\begin{theorem}\label{Franks}
Given $\epsilon>0$ and a vector field $X\in\mathfrak{X}_{\mu}^4(M)$, there exists $\xi_0=\xi_0(\epsilon,X)$ such that for any $\tau\in\left[1,2\right]$, for any periodic point $p$ of period greater than $2$, for any sufficient small flowbox $\mathcal{T}$ of $X^{\left[0,\tau\right]}(p)$ and for any one-parameter linear family $\left\{A_t\right\}_{t\in\left[0,\tau\right]}$ such that $\|A_t'A_t^{-1}\|<\xi_0$, for all $\:t\in\left[0,\tau\right]$, there exists $Y\in\mathfrak{X}_{\mu}^1(M)$ satisfying the following properties:
\begin{enumerate}
	\item $Y$ is $\epsilon$-$C^1$-close to $X$;
	\item $Y^t(p)=X^t(p)$, $\forall\: t \in\mathbb{R}$;
	\item $P_Y^{\tau}(p)=P_X^{\tau}(p)\circ A_{\tau}$;
	\item $Y|_{\mathcal{T}^c}=X|_{\mathcal{T}^c}$.
\end{enumerate}
\end{theorem}

In the sequence, we state a version of the $C^1$-Closing Lemma for volume-preserving flows, firstly proved by Pugh and Robinson (see \cite{PR}) and that, more recently, was improved by Arnaud, that presented a simpler proof (see \cite{Ar}). It states that the orbit of a recurrent point can be approximated by a long time closed orbit of a $C^1$-perturbation of the original vector field.

\begin{theorem}\label{closing}
Take $X\in \mathfrak{X}^1_{\mu}(M)$ and $x$ a $X^t$-recurrent point. Given $\epsilon, r, T>0$, there is an $\epsilon$-$C^1$-neighborhood $\mathcal{U}\subset\mathfrak{X}^1_{\mu}(M)$ of $X$, a closed orbit $p$ of $Y\in\mathcal{U}$ with period $\pi$ arbitrarily large, a map $g:[0,T]\rightarrow[0,\pi]$ close to the identity and $ \tilde{T}>T$ such that
\begin{itemize}
	\item $dist\big(X^t(x),Y^{g(t)}(p)\big)<\epsilon$, for every $0\leq t\leq \tilde{T}$;
	\item $Y=X$ on $M\backslash B_{r}\big(X^{[0,\tilde{T}]}(x)\big)$.
\end{itemize}
\end{theorem}

A conservative version of Pugh and Robinson's \textit{General Density Theorem} (see \cite{PR}), also proved by Arnaud in \cite{Ar}, asserts that, $C^1$-generically, the closed orbits are dense in $M$. 
We denote by $\mathcal{PR}^1_{\mu}(M)$ this residual set in $\mathfrak{X}^1_{\mu}(M)$.

The next result correspond to a dichotomy for conservative vector fields. It requires the existence of a closed orbit with arbitrarily large period and it is obtained following the ideas presented on \cite[Proposition 2.4]{BR}.

\begin{theorem}\label{BGV}
Let $X\in\mathfrak{X}^1_{\mu}(M)$ and let $\mathcal{U}$ be a small $C^1$-neighborhood of $X$. Then, for any $\epsilon>0$, there exist $l,\tau>0$ such that, for any $Y\in \mathcal{U}$ and any closed orbit $x$ of $Y^t$ of period $\pi(x)>\tau$,
\begin{itemize}
	\item either ${P}_Y^t$ admits an $l$-dominated splitting over the $Y^t$-orbit of $x$, or else
  \item for any neighborhood $U$ of $x$, there exists an $\epsilon$-$C^1$-perturbation $\tilde{Y}$ of $Y$, coinciding with $Y$ outside $U$ and along the orbit of $x$, such that ${P}^{\pi{(x)}}_{\tilde{Y}}(x)=id$.
\end{itemize}
\end{theorem}

Take $X\in \mathfrak{X}^1_{\mu}(M)$. By Oseledets's theorem (see \cite{O}), $\mu$-almost every point $x$ in $M$ has a splitting of the tangent bundle, $T_xM=E_x^1\oplus\cdot\cdot\cdot\oplus E_x^{k(x)}$, called the \textit{Oseledets splitting}, and real numbers $\lambda_1(x)>\cdot\cdot\cdot> \lambda_{k(x)}(x)$, called the \textit{Lyapunov exponents}, $1\leq k(x)\leq n$, such that $DX^t_x(E^i_x)=E^i_{X^t(x)}$ and $$\lambda_i(x)=\displaystyle \lim_{t\rightarrow \pm\infty}\frac{1}{t}\log \|DX^t_x\:(v^i)\|,$$ for any $v^i\in E^i_x\setminus\{\vec{0}\}$ and  $i\in\{1,...,k(x)\}$.
The full $\mu$-measure set of the \textit{Oseledets points} is denoted by $\mathcal{O}(X)$.

\begin{remark}\label{remsum0}
As a consequence of Oseledets's theorem one has that $$\displaystyle\sum_{i=1}^{k(x)}\lambda_i(x)\cdot\dim(E^i_x)=\lim_{t\rightarrow\pm\infty}\dfrac{1}{t}\log|\det DX^t_x|.$$

However, since the vector field $X$ is divergence-free, we deduce that $|\det DX^t(x)|=1$, for every $t\in\mathbb{R}$ and every $x\in M$.
So, we conclude that $$\displaystyle\sum_{i=1}^{k(x)}\lambda_i(x)\cdot\dim(E^i_x)=0, \;\;\forall \: x\in \mathcal{O}(X).$$

Note that if we do not take into account the multiplicities of the eigenvalues associated to the eigenspaces $E_x^1, \cdot\cdot\cdot,E_x^{k(x)}$, we have exactly $n$ Lyapunov exponents $\lambda_1(x)\geq\cdot\cdot\cdot\geq \lambda_{n}(x)$.
\end{remark}

If we assume the absence of a dominated splitting, it is possible to make a $C^1$-perturbation of the vector field in order to get a new one with Lyapunov exponents arbitrarily close to zero, as it is shown in \cite[Theorem 1]{BR0}. 

Now, we state that a singularity $p$ is \textit{linear} if there exist smooth local coordinates around $p$ such that $X$ is linear and equal to $DX(p)$ in these coordinates (cf. \cite[Definition 4.1]{V1}).
The next lemma states that any singularity can be turned into a linear one, by performing a small perturbation of the vector field.

\begin{lemma}\label{linearsing}
If $X\in\mathfrak{X}^1_{\mu}(M)$ has a singularity then, for any neighbourhood $\mathcal{V}$ of $X$, there is an open and nonempty set $\mathcal{U}\subset\mathcal{V}$ such that any $Y\in\mathcal{U}$ has a linear hyperbolic singularity.
\end{lemma} 

\begin{proof}
Let $p$ be a singularity of $X\in\mathfrak{X}^1_{\mu}(M)$ and $\epsilon>0$.
By a small $C^1$-conservative perturbati\-on of $X$ (see \cite{BR}), we can find $X_1$, $\epsilon$-$C^1$-close to $X$, with a hyperbolic singularity $p$. 
Denote by $\mathcal{V}$ a $C^1$-neighbourhood of $X_1$ in $\mathfrak{X}^1_{\mu}(M)$ where the analytic continuation of $p$ is well-defined.
Now, by Zuppa's Theorem (see \cite{Z}), there is a smooth vector field $X_2\in\mathcal{V}$ with a hyperbolic singularity $p_2$. 
If the eigenvalues of $DX_2(p_2)$ satisfy the nonresonance conditions of the Sternberg linearization theorem (see \cite{SS}) then there is a smooth diffeomorphism conjugating $X_2$ and its linear part around $p_2$. 
If the nonresonance conditions are not satisfied then we can perform a $C^1$-conservative perturbation of $X_2$, so that the eigenvalues satisfy the nonresonance conditions.
So, since the set of divergence-free vector fields satisfying the nonresonance conditions is an open and dense set in $\mathfrak{X}^1_{\mu}(M)$, there is a $C^1$-neighbourhood ${\mathcal{U}}$ of $X_2$ in $\mathcal{V}$ such that any vector field $X_3\in{\mathcal{U}}$ is conjugated to its linear part, meaning that $X_3$ has a linear hyperbolic singularity.
\end{proof}

The next result, will be used to prove that a star vector field can not exhibit singularities.

\begin{theorem}\cite[Proposition 4.1]{V1}\label{vivier}
If $X\in\mathfrak{X}^1(M)$ admits a linear hyperbolic singularity of saddle-type then the ${P}_X^t$ does not admit any dominated splitting over $M\backslash Sing(X)$.
\end{theorem} 

The final presented auxiliary result asserts that, $C^1$-generically, a vector field is topologically mixing, and so transitive.

\begin{theorem}\cite[Theorem 1.1]{B1}\label{topmix}
There exists a $C^1$-residual subset $\mathcal{R}\subset\mathfrak{X}^1_{\mu}(M)$ such that, if $X\in\mathcal{R}$ then $X$ is a topologically mixing vector field.
\end{theorem} 

\end{section}

\begin{section}{Auxiliary lemmas}\label{auxlemmas}

We start this section by showing that a divergence-free star vector field does not have singularities.

\begin{lemma}\label{Gnosing}
If $X\in\mathcal{G}^1_{\mu}(M)$ then $Sing(X)=\emptyset$.
\end{lemma}

\begin{proof}
Fix $X\in\mathcal{G}^1_{\mu}(M)$ and $\mathcal{U}$ a $C^1$-neighborhood of $X$ in $\mathcal{G}^1_{\mu}(M)$, small enough such that Theorem \ref{BGV} holds.
Recall that $\mathcal{PR}^1_{\mu}(M)$ is a residual set in $\mathfrak{X}^1_{\mu}(M)$ such that any $X\in\mathcal{PR}^1_{\mu}(M)$ has the closed orbits dense in $M$ (see \cite[\S 8(c)]{PR}). 
Let $\mathcal{R}$ be the residual set given by Theorem \ref{topmix} 

To obtain a contradiction, take $p\in Sing(X)$, which is hyperbolic and of saddle-type, by definition of $\mathcal{G}^1_{\mu}(M)$, and so it persists to $C^1$-small perturbations of $X$.
By Lemma\:\ref{linearsing}, there is $Y\in\mathcal{U}\cap\mathcal{R}\cap\mathcal{PR}^1_{\mu}(M)$, $C^1$-close to $\tilde{X}$, such that $p\in Sing(Y)$ is linear hyperbolic of saddle-type, and $Y$ has a closed orbit $x$, with arbitrarily large period. 
Therefore, as $Y\in\mathcal{G}^1_{\mu}(M)$, by Theorem~\ref{BGV}, there exist constants $\ell,\tau>0$ such that ${P}_Y^t$ admits an $\ell$-dominated splitting over the $Y^t$-orbit of $x$ with period $\pi(x)>\tau$.
Also, $Y$ has a dense orbit because it belongs to $\mathcal{R}$. So, by the volume preserving Closing Lemma (Theorem~\ref{closing}), there is a sequence of vector fields $Y_n\in\mathcal{U}\cap\mathcal{R}$, $C^1$-converging to $Y$, and, for every $n\in\mathbb{N}$, $Y_n$ has a closed orbit $\Gamma_n=\Gamma_n(t)$ of period $\pi_n$ such that $\displaystyle\lim_{n\rightarrow\infty}\Gamma_n(0)=x$ and $\displaystyle\lim_{n\rightarrow\infty}\pi_n=+\infty$. 
Therefore, by Theorem~\ref{BGV}, ${P}_{Y_n}^t$ admits an $\ell$-dominated splitting over the orbit $\Gamma_n$, for large $n$. Taking a subsequence if necessary, say $n\in I\subseteq\mathbb{N}$, we have a sequence of $Y_n$ with closed orbit $\Gamma_n$ such that ${P}_{Y_n}^t$ has an $\ell$-dominated splitting and such that the dimensions of the invariant bundles do not depend on $n$. 
Then, given that $$M=\displaystyle\limsup_{n}\Gamma_n=\displaystyle\bigcap_{N\in\mathbb{N}}\biggl(\overline{\bigcup_{n\geq N}^{\infty}\Gamma_n}\biggr),$$ we conclude that there exists an $\ell$-dominated splitting for ${P}_Y^t$ over $M\backslash Sing(Y)$.

However, since $p$ is a linear hyperbolic singularity of saddle-type of $Y$, by Theorem \ref{vivier}, we conclude that ${P}_Y^t$ does not admit a dominated splitting over $M\backslash Sing(Y)$. This is a contradiction. So, $X$ has no singularities. 
\end{proof}


The next lemma states that, given a divergence-free star vector field, we can define a continuous splitting $N=N^1\oplus N^2$ over $M$.

\begin{lemma}\label{splitlemma}
If $X\in\mathcal{G}^1_{\mu}(M)$ then there exists a continuous splitting $N_x=N^1_x\oplus N^2_x$, for every $x\in M$.
\end{lemma}

\begin{proof}
Take $X\in\mathcal{G}^1_{\mu}(M)$ and recall that, by Lemma \ref{Gnosing}, $Sing(X)=\emptyset$. So, we have that $N^1_p=N^s_p$ and $N^2_p=N^u_p$, for any $p\in Per(X)$. To extend these fibers to any $x\in M$, fix $y\notin Per(X)$ and a sequence $\{y_n\}_n\in Per(X)$ such that $\displaystyle\lim_{n\rightarrow+\infty} y_n=y$ and $\displaystyle\lim_{n\rightarrow+\infty} N^{{2},{1}}_{X^t(y_n)}= N^{{u},{s}}_{X^t(y)}$. So, any $x\in M$ has attached the subspaces $N^{{2},{1}}_x$ such that 
\begin{align}\label{dim}
\dim N^1_x+\dim N^2_x=\dim M-1 
\end{align}
and $P^t_X(x)( N^{{2},{1}}_x)= N^{{2},{1}}_{X^t(x)}$.
Notice that, by \cite[Lemma 3.1]{M1}, since \linebreak$X\in\mathcal{G}^1_{\mu}(M)$ then $\overline{Per(X)}=\Omega(X)=M$. So, the domination over $Per(X)$, that can be extended to $\overline{Per(X)}=M$, leads to $N^1_x\cap N^2_x=\left\{0\right\}$, for any $x\in M$. This with (\ref{dim}) implies that $N_x=N_x^1\oplus N_x^2$ and that the fibers depend continuously on $x$, for any $x\in M$.
\end{proof}

The proof of the next lemma uses a generalization, for the higher-di\-men\-sio\-nal context, of the adopted techniques in the proof of Lemma~3.1 in \cite{BR2}. However, at this point, we already know that a vector field in $\mathcal{G}^1_{\mu}(M)$ has not singularities.

\begin{lemma}\label{Gdomsplit}
If $X\in\mathcal{G}^1_{\mu}(M)$ then ${P}_X^t$ admits a dominated splitting over $M$.
\end{lemma}

\begin{proof}
Take $X\in\mathcal{G}^1_{\mu}(M)$ and $\mathcal{U}$ a $C^1$-neighborhood of $X$ in $\mathcal{G}^1_{\mu}(M)$, small enough such that Theorem \ref{BGV} holds. 
By Lemma \ref{Gnosing}, we have that $M$ is regular for $X$. 
So, $P_X^t$ is well defined on $M$.
It follows that there exists $\mathcal{V}\subset\mathcal{U}$, a $C^1$-neighborhood of $X$ in $\mathcal{G}^1_{\mu}(M)$, whose elements do not have singularities. 
By Lemma \ref{splitlemma}, we have a continuous splitting $N=N^1\oplus N^2$ over $M$.
By contradiction, assume that this splitting is not dominated. 
So, we claim that
\begin{claim}
For all $\ell\in\mathbb{N}$, there exists an $X^t$-invariant and measurable set $\Gamma_{\ell}\in M$ such that $\mu(\Gamma_{\ell})>0$ and $\Gamma_{\ell}$ does not have an  $\ell$-dominated splitting for ${P}_X^t$.
\end{claim}
In fact, if the claim was not true, there would exist $\ell\in\mathbb{N}$ such that $M$ has an $\ell$-dominated splitting for ${P}_X^t$, which contradicts our assumption.
The existence of these sets $\Gamma_{\ell}$ without an $\ell$-dominated splitting, for any $\ell\in\mathbb{N}$, allow us to use the techniques involved in the proof of \cite[Theorem 1]{BR0} in order to conclude that, for any $\epsilon>0$, there exists $\ell\in\mathbb{N}$, large enough, such that, for any $\eta>0$ arbitrarily small and for $\mu$-almost every point $x\in\Gamma_{\ell}$, we can find $t_0>0$ and $X_1\in\mathcal{U}$, $\epsilon$-$C^1$-close to $X$, satisfying $$exp({-\eta t})<\|P^t_{X_1}(x)\|<exp({\eta t}), \ \ \forall\: t>t_0.$$

Now, let $R\subset\Gamma_{\ell}$ be the full $\mu$-measure set of recurrent points, given by the Poincar\'{e} recurrence theorem with respect to $X_1$, and let $\mathcal{Z}_{\eta}\subset\Gamma_{\ell}$ be the set of points with Lyapunov exponent, associated to $X_1$, less than $\eta$.

So, fixing $\delta\in\big(0, \frac{\log2}{(n-1)\ell}\big)$ and $\eta<\delta$, given $x\in \mathcal{Z}_{\eta}\cap R$, there exists $t_x\in\mathbb{R}$ such that $$exp({-\delta t})<\|P^t_{X_1}(x)\|<exp({\delta t}), \ \ \forall\: t>t_x,$$ where we can assume that $t_x\geq T$.

Now, once $x\in\mathcal{Z}_{\eta}\cap R$, by the volume preserving Closing Lemma (Theorem \ref{closing}), the $X_1^t$-orbit of $x$ can be approximated by a closed orbit $p$ with period $\pi$ of a $C^1$-close vector field $X_2\in\mathcal{U}$. So, letting $r>0$ be small enough in Theorem \ref{closing}, we can take $\pi>\max\{\tau, T\}$ arbitrarily large, where $\tau>0$ is given by Theorem \ref{BGV}, and is such that 
\begin{align}
\exp(-\delta\pi)<\|P_{X_2}^{\pi}(p)\|<\exp(\delta\pi).\label{bound}
\end{align}

Note that $X_2\in\mathcal{U}$, a $C^1$-neighborhood of $X$ in $\mathcal{G}^1_{\mu}(M)$, and that $p$ is a $X_2$-closed orbit with period $\pi>\tau$, obviously hyperbolic. 
So, by Theorem \ref{BGV}, there is $\ell_0>0$ such that $P^t_{X_2}$ admits an $\ell_0$-dominated splitting $N_q=N^1_q\oplus\cdots\oplus N^k_q$, $2\leq k\leq n-1$, such that $$\|P^{\ell_0}_{X_2}(q)|_{N^i_q}\|\cdot \|P^{-\ell_0}_{X_2}(q)|_{N^j_q}\| \leq\dfrac{1}{2},$$ for every $0\leq i<j\leq k$ and every $q\in\mathcal{O}_{X_2}(p)$.

Now, given that $p$ is a hyperbolic saddle with period $\pi$ for $X_2$, let us assume that $P^{\pi}_{X_2}(p)$ admits the following Lyapunov spectrum: $$\lambda_1(p)\geq ...\geq\lambda_r(p)>0>\lambda_{r+1}(p)\geq...\geq\lambda_{k}(p).$$ So, let $N_p^u=N_p^1\oplus\cdots\oplus N_p^r$ and $N^s_p=N_p^{r+1}\oplus\cdots\oplus N_p^k$.

Let $[a]$ denote the integer part of $a$ and observe that
\begin{align}
&\|P^{\pi}_{X_2}(p)|_{N_p^s}\|\cdot\|P^{-\pi}_{X_2}(p)|_{N_p^u}\|\nonumber\\
&=\|P^{\pi-{\ell_0}[\pi/{\ell_0}]+{\ell_0}[\pi/{\ell_0}]}_{X_2}(p)|_{N_p^s}\|\cdot  \|P^{-\pi-{\ell_0}[\pi/{\ell_0}]+{\ell_0}[\pi/{\ell_0}]}_{X_2}(p)|_{N_p^u}\|\nonumber\\
&\leq \|P^{\pi-{\ell_0}[\pi/{\ell_0}]}_{X_2}(p)|_{N_p^s}\|\cdot \|P^{{\ell_0}[\pi/{\ell_0}]}_{X_2}(X_2^{{\ell_0}[\pi/{\ell_0}]}(p))|_{N_{X_2^{{\ell_0}[\pi/{\ell_0}]}(p)}^s}\|\cdot \nonumber\\
&\ \ \ \ \ \cdot \|P^{-\pi+{\ell_0}[\pi/{\ell_0}]}_{X_2}(p)|_{N_p^u}\|\cdot \|P^{-{\ell_0}[\pi/{\ell_0}]}_{X_2}(X_2^{-{\ell_0}[\pi/{\ell_0}]}(p))|_{N_{X_2^{-{\ell_0}[\pi/{\ell_0}]}(p)}^u}\| \nonumber\\
&\leq C(p, X_2)\: \displaystyle\prod_{i=1}^{[\pi/{\ell_0}]} \|P^{{\ell_0}}_{X_2}(X_2^{{\ell_0}}(p))|_{N_{X_2^{i{\ell_0}}(p)}^s}\|\cdot \|P^{-{\ell_0}}_{X_2}(X_2^{-{\ell_0}}(p))|_{N_{X_2^{-i{\ell_0}}(p)}^u}\|\nonumber\\
&\leq C(p, X_2)\: \bigg(\dfrac{1}{2}\bigg)^{[\pi/{\ell_0}]},\nonumber
\end{align}
where $C(p, X_2)=\displaystyle \sup_{0\leq t \leq {\ell_0}}\Big(\|P^{t}_{X_2}(p)|_{N_p^s}\| \cdot \|P^{-t}_{X_2}(p)|_{N_p^u}\|\Big)$. Once $C(p, X_2)$ depends continuously on $X_2$, in the $C^1$-topology, there exists a uniform bound for $C(p,\cdot)$, for every vector field which is $C^1$-close to $X_2$.

As it was mentioned in Remark \ref{remsum0}, we have that $\displaystyle\sum_{i=1}^{k}\lambda_i(p)=0$. So, observing that $\|P^{\pi}_{X_2}(p)\|=\|P^{\pi}_{X_2}(p)|_{N_p^2}\|$, one has that
\begin{align}
\dfrac{1}{\pi}\log\|P^{\pi}_{X_2}(p)|_{N_p^1}\|&=\lambda_{r+1}(p)=-\displaystyle\sum_{\stackrel{i=1}{i\neq r+1}}^{k}\lambda_i(p)\nonumber\\
&\geq -(k-1)\lambda_1(p)= \dfrac{-(k-1)}{\pi}\log\|P^{\pi}_{X_2}(p)|_{N_p^u}\|\nonumber\\
&=\dfrac{-(k-1)}{\pi}\log\|P^{\pi}_{X_2}(p)\|\nonumber.
\end{align}

So, from $\|P^{\pi}_{X_2}(p)|_{N_p^s}\|  \|P^{-\pi}_{X_2}(p)|_{N_p^u}\| \leq C(p, X_2)\: \bigg(\dfrac{1}{2}\bigg)^{[\pi/{\ell_0}]}$
and knowing that $\|P^{\pi}_{X_2}(p)|_{N_p^u}\|^{-1}\leq \|P^{-\pi}_{X_2}(p)|_{N_p^u}\|$, we have that
\begin{align}
&\|P^{\pi}_{X_2}(p)|_{N_p^s}\|  \|P^{\pi}_{X_2}(p)|_{N_p^u}\|^{-1} \leq C(p, X_2)\: \bigg(\dfrac{1}{2}\bigg)^{[\pi/{\ell_0}]}\nonumber\\
\Leftrightarrow &\log\|P^{\pi}_{X_2}(p)|_{N_p^s}\|-\log \|P^{\pi}_{X_2}(p)|_{N_p^u}\|\leq \log C(p, X_2) -[\pi/{\ell_0}]\log 2\nonumber\\
\Leftrightarrow & \dfrac{1}{\pi} \log \|P^{\pi}_{X_2}(p)\| \geq -\dfrac{\log C(p, X_2)}{\pi} +\dfrac{[\pi/{\ell_0}]\log 2}{\pi}  + \dfrac{1}{\pi} \log\|P^{\pi}_{X_2}(p)|_{N_p^s}\|\nonumber\\
\Leftrightarrow & \dfrac{1}{\pi}\log \|P^{\pi}_{X_2}(p)\|  \geq -\dfrac{\log C(p, X_2)}{\pi} +\dfrac{[\pi/{\ell_0}]\log 2}{\pi}\nonumber\\
& \ \ \ \ \ \ \ \ \ \ \ \ \ \ \ \ \ \ \ \ \ \ \ - \dfrac{(k-1)}{\pi}\log\|P^{\pi}_{X_2}(p)\|\nonumber.
\end{align}

Now, taking $\pi$ arbitrarily large,

$$\dfrac{1}{\pi}\log \|P^{\pi}_{X_2}(p)\|  \geq \dfrac{\log2}{k{\ell_0}}\geq \dfrac{\log2}{(n-1){\ell_0}} >\delta.$$

But this contradicts (\ref{bound}). Then ${P}_X^t$ admits a dominated splitting $N=N^1\oplus N^2$ over $M$.
\end{proof}

\begin{remark}\label{remlemma}
Notice that the previous lemma is also true if we suppose that $X$ is an isolated point in the boundary of $\mathcal{A}_{\mu}(M^n)$, for fixed $n\geq 4$. In fact, in Lemma \ref{Gdomsplit} we need $X\in\mathcal{G}^1_{\mu}(M^n)$ in order to ensure a dominated splitting over a closed orbit $x$, with large enough period $\pi$, for a vector field $Y$, $C^1$-close to $X$, given by Theorem \ref{BGV}. However, if we start the proof by assuming that $X$ is an isolated point in the boundary of $\mathcal{A}_{\mu}(M^n)$, we must obtain the same conclusion, because any $C^1$-perturbation $\tilde{Y}$ of $Y$ must be Anosov, and so cannot satisfy $P^{\pi}_{\tilde{Y}}(x)=id$.
\end{remark}


Next lemma is an adaption of the ideas of Ma\~{n}\'{e} (\cite {M}) to our setting.

\begin{lemma}\label{mane1}
Take $X\in\mathfrak{X}_{\mu}^1(M)$ and assume that $M$ is regular and that any $x\in M$ admits a dominated splitting $N_x=N_x^1\oplus N_x^2$. 
If $\displaystyle\liminf_{t\rightarrow\infty}\|P_X^t(x)|_{N^1_x}\|=0$ and $\displaystyle\liminf_{t\rightarrow\infty}\|P_X^{-t}(x)|_{N^2_x}\|=0$, for all $x\in M$, then $M$ is hyperbolic.
\end{lemma}

\begin{proof}
By hypothesis, for any $x\in M$ we can find $t_x$ such that \linebreak $\|P_X^{t_x}(x)|_{N^1_x}\|<1/3$ and so, every $x\in M$ has a neighborhood $B(x)$ such that every $y\in B(x)$ satisfies $\|P_X^{t_x}(y)|_{N^1_{y}}\|<1/2$.

Since $M$ is compact, there are $x_1,...,x_n\in M$, such that $M\subset \displaystyle\bigcup_{i=1}^n B(x_i)$. So, given any $y\in M$, there is $1\leq i_1\leq n$ such that $y\in B(x_{i_1})$.

Let $K_0=\sup\{\|P_X^t(y)|_{N^1_y}\|, y\in B(x_i), 0\leq t \leq t_{x_i}, 1\leq i\leq n\}$, $j_0$ be such that ${K_0}/{2^{j_0}}<{1}/{2}$ and $T_0>j_0\sup\{t_{x_i}, 1\leq i\leq n\}$. Let us see that $T_0$ is the uniform hyperbolicity constant.

Take $t_{i_1},...,t_{i_{k+1}}$ and $l_j=t_{i_1}+...+t_{i_j}$, $1\leq j\leq k+1$, such that

(1) $X^{l_j}(y)\in B(x_{i_{j+1}}), 1\leq j\leq k$,

(2) $l_k\leq T_0 \leq l_{k+1}$.

From the previous, we observe that $k\geq j_0$ and $0\leq T_0-l_k\leq t_{i_{k+1}}$. 

So, for any $y\in M$,
\begin{align}
\|P_X^{T_0}(y)|_{N^1_y}\|&=\|P_X^{T_0-l_k+l_k}(y)|_{N^1_y}\|\nonumber \\
&\leq\|P_X^{T_0-l_k}(X^{l_k}(y))|_{N^1_{X^{l_k}(y)}}\|   \cdot  \|P_X^{t_{i_1}}(y)|_{N^1_y}\|\nonumber \cdot\\ 
&
\ \ \ \ \cdot\|P_X^{t_{i_2}}(X^{l_1}(y))|_{N^1_{X^{l_1}(y)}}\|\cdots \|P_X^{t_{i_k}}(X^{l_{k-1}}(y))|_{N^1_{X^{l_{k-1}}(y)}}\|\nonumber \\
&\leq K_0 \;\dfrac{1}{2^k}\leq K_0 \;\dfrac{1}{2^{j_0}}<\dfrac{1}{2}.\nonumber
\end{align}
Changing $P^t_X$ by $P^{-t}_X$, the second case can be derived from this one.

\end{proof}

\end{section}

In the following lemma, we show that a divergence-free star vector field has uniform hyperbolicity in the period, which is a crucial step to derive hyperbolicity from Lemma \ref{Gdomsplit}.

\begin{lemma}(Uniform hyperbolicity in the period)\label{unifhyper}
Fix $X\in\mathcal{G}^1_{\mu}(M)$. There exist $\mathcal{U}$, a $C^1$-neighborhood of $X$ on $\mathcal{G}^1_{\mu}(M)$, and $\theta\in (0,1)$ such that, for any $Y\in\mathcal{U}$, if $p\in Per(Y)$ has period $\pi_Y(p)$ and has the hyperbolic splitting $N_p=N^s_p\oplus N_p^u$ then:
\begin{enumerate}
	\item [(a)]  $\|P_Y^{\pi_Y(p)}(p)|_{N^s_{p}}\|<\theta^{\pi_Y(p)}$  and
	\item [(b)] $\|P_Y^{-\pi_Y(p)}(p)|_{N^u_{p}}\|<\theta^{\pi_Y(p)}$. 
\end{enumerate}
\end{lemma}

\begin{proof}
Take $X\in\mathcal{G}^1_{\mu}(M)$ and $\mathcal{U}$ a $C^1$-neighborhood of $X$ in $\mathcal{G}^1_{\mu}(M)$. So, for every $p\in Per(Y)$ with period $\pi_Y(p)$, where $Y\in\mathcal{U}$, we have that $p$ is a hyperbolic saddle, meaning that $N_p=N^s_p\oplus N^u_p$ and that there is a constant $\theta_p\in (0,1)$ such that $\|{P}_Y^{\pi_Y(p)}(p)|_{{N}^s_{p}}\|<\theta_p^{\pi_Y(p)}$ and $\|{P}_Y^{-\pi_Y(p)}(p)|_{{N}^u_{p}}\|<\theta_p^{\pi_Y(p)}$. However, we want $\theta_p$ to be uniform. 

Let us prove (a). Suppose that, by contradiction, for any $\theta\in (0,1)$ there exists $Y\in\mathcal{U}$, $C^1$-arbitrarily close of $X$, and $p\in Per(Y)$ with period $\pi_Y(p)$, hyperbolic by hypothesis, such that $$\theta^{\pi_Y(p)}\leq\|{P}_Y^{\pi_Y(p)}(p)|_{{N}^s_{p}}\|.$$

In order to apply Theorem \ref{Franks}, we must have a $C^4$-vector field.
So, using Zuppa's theorem (Theorem \ref{zuppa}), we start by $C^1$-approximate $Y$ by a vector field $\tilde{Y}\in\mathcal{U}\cap\mathfrak{X}^4_{\mu}(M)$, such that $\tilde{p}\in Per(\tilde{Y})$ is the hyperbolic continuation of $p$, so with period $\pi_{\tilde{Y}}(\tilde{p})$ close to $\pi_Y(p)$, and

\begin{align}
\theta^{\pi_{\tilde{Y}}(\tilde{p})}\leq\|{P}_{\tilde{Y}}^{\pi_{\tilde{Y}}(\tilde{p})}(\tilde{p})|_{{N}^s_{\tilde{p}}}\|.\label{sigmaexp}
\end{align}
For simplicity, assume that $\pi_{\tilde{Y}}(\tilde{p})$ is an integer.
By (\ref{sigmaexp}), $\theta\leq\|{P}_{\tilde{Y}}^{1}(\tilde{q})|_{{N}^s_{\tilde{q}}}\|$, for some $\tilde{q}\in\mathcal{O}_{\tilde{Y}}(\tilde{p})$.

Let $A_t$ be a one-parameter family of linear maps, for $t\in \left[0,\pi_{\tilde{p}}\right]$, such that $\|A'_tA^{-1}_t\|$ is arbitrarily small, for any $t\in \left[0,1\right]$ and suppose $\|P^1_{\tilde{Y}}{(\tilde{q})}|_{{N}^s_{\tilde{q}}}\|=1-\gamma,$  
where, by expression (\ref{sigmaexp}), $\gamma$ is such that $0~<~\gamma~<~1-\theta$ and $\theta$ is chosen arbitrarily close to $1$.
Now, take $A_{t}=id$, for $t\leq 0$, and $A_t$ a homothetic transformation of ratio of order $\dfrac{1}{1-\gamma}$, for $t\in\left[0,\pi_{\tilde{Y}}(\tilde{p})\right]$, and with entry $a_{1,n-1}=\delta\alpha(t)$, where $\alpha(t)$ is a smooth function such that $\alpha(t)=1$, for $t\geq 1$, $\alpha(t)=0$, for $t\leq 0$, $0<\alpha'(t)<1$, and $\delta>0$ is arbitrarily small. It is straightforward to see that $\|A'_{t}A^{-1}_{t}\|<\frac{\delta}{1-\gamma}$ and that this norm can be taken arbitrarily small, by choosing $\delta>0$ small enough. 

Take $\epsilon>0$ and divide $\pi_{\tilde{Y}}(\tilde{p})=\pi_{\tilde{Y}}(\tilde{q})$ in $\pi_{\tilde{Y}}(\tilde{q})$-one-time intervals. 
By\linebreak Theorem~\ref{Franks}, there exist vector fields $Z_i\in\mathcal{G}^1_{\mu}(M)$, $\dfrac{\epsilon}{\pi_{\tilde{Y}}(\tilde{q})}$-$C^1$-close to $\tilde{Y}$, such that 
$P_{Z_i}^1(\tilde{q})=P^1_{\tilde{Y}}(\tilde{q})\circ A_{1}$, for $i\in\{1,...,\pi_{\tilde{Y}}(\tilde{q})\}$. 
So, by Theorem \ref{pasting}, there exists $Z\in\mathcal{G}^1_{\mu}(M)$, $\epsilon$-$C^1$-close to $\tilde{Y}$, such that $P_Z^{\pi_{\tilde{Y}}(\tilde{q})}(\tilde{q})$ has a eigenvalue equal to $1$ or $-1$. 
This is a contradiction because, since $Z\in\mathcal{G}^1_{\mu}(M)$, $\tilde{q}$ is a hyperbolic closed orbit of saddle-type and so its spectrum must be disjoint from $\mathbb{S}^1$. So, (a) must hold.

Using a similar argument, (b) is proved. This finishes the proof.
\end{proof}


\begin{section}{Proof of Theorem \ref{mainth}}\label{th1}

In this section, we conclude the proof of Theorem \ref{mainth}, by adapting to our setting a technique due to Ma\~{n}\'{e} in \cite{M}.

Take ${X}\in \mathcal{G}^1_{\mu}(M)$. 
By Lemma \ref{Gnosing}, Lemma \ref{splitlemma} and Lemma \ref{Gdomsplit}, we have that $M$ is regular for $X$ and that $P_X^t$ admits a dominated splitting $N=N^1\oplus N^2$ over $M$. 
We want to prove that $P^t_X|_{N^1}$ is uniformly contracting on $M$ and that $P^t_X|_{N^2}$ is uniformly expanding on $M$. 
Let us prove the first condition. 
By Lemma \ref{mane1}, it suffices to prove that $$\displaystyle\liminf_{t\rightarrow\infty}\|P_X^t(x)|_{N^1_x}\|=0, \:\forall \:x\in M.$$
By contradiction, suppose that there is $x\in M$ such that $$\displaystyle\liminf_{t\rightarrow\infty}\|P_X^t(x)|_{N^1_x}\|>0.$$ 
Then, we can choose a subsequence $\{t_n\}_{n\in\mathbb{N}}$ such that $t_n\rightarrow\infty$ as $n\rightarrow\infty$ and 
\begin{align}
\displaystyle\lim_{n\rightarrow\infty}\dfrac{1}{t_n}\log\|P_X^{t_n}(x)|_{N^1_x}\|\geq 0.\label{exp1}
\end{align}

Let $C(M)$ be the set of continuous functions on $M$ and define \linebreak$\varphi:C(M)\rightarrow \mathbb{R}$ by $\varphi(p)=\partial_h(\log\|P_X^h(p)|_{N_p^1}\|)_{h=0}$. By the Riez Theorem, there is a $X^t$-invariant Borel probability measure $\mu$ such that 

\begin{align}
\int_M\varphi \:d\mu&=\lim_{t_n\rightarrow+\infty} \frac{1}{t_n}\int_0^{t_n}\varphi(X^s(x))\:ds\nonumber\\
&=\lim_{t_n\rightarrow+\infty} \frac{1}{t_n}\int_0^{t_n} \partial_h(\log\|P_X^h(X^s(x))|_{N_{X^s(x)}^1}\|)_{h=0}\:ds\nonumber\\
&=\lim_{t_n\rightarrow+\infty} \frac{1}{t_n} \log \|P_X^{t_n}(x)|_{N_{x}^1}\|\geq0.\nonumber
\end{align}
Also, by the Birkhoff Ergodic Theorem, $$\int_M\varphi \:d\mu=\int_M\lim_{t\rightarrow+\infty} \frac{1}{t}\int_0^{t}\varphi(X^s(x))\:dsd\mu(x)\geq0.$$

Now, let $\Sigma(X)$ be the set of points $x\in M$ such that, for any $C^1$-neigh\-bour\-hood $\mathcal{U}$ of $X$ in $\mathfrak{X}^1_{\mu}(M)$ and $\delta>0$, there exist $Y\in \mathcal{U}$ and a $Y$-closed orbit $y\in M$ of period $\pi$ such that $X=Y$ except on the $\delta$-neighborhood of the $Y$-orbit of $y$, and that $dist(Y^t(y),X^t(x))<\delta$, for $0\leq t\leq \pi$. 
A conservative version of the Ergodic Closing Lemma, proved by Arnaud in \cite{Ar}, says that, given a $X^t$-invariant Borel probability measure $\mu$, $\mu(\Sigma(X))=1$. So, there is $x\in\Sigma(X)$ such that
\begin{align}\label{expansion}
\lim_{t\rightarrow+\infty} \frac{1}{t}\int_0^{t}\varphi(X^s(x))\:ds=\lim_{t\rightarrow+\infty} \frac{1}{t}\log\|P_X^{t}(x)|_{N_{x}^1}\|\geq0.
\end{align}
Let $\log\theta<\delta<0$ be arbitrarily small, where $\theta\in(0,1)$ is fixed and given by Lemma \ref{unifhyper}. Thus, there is $t_{\delta}$ such that, for $t\geq t_{\delta}$, 
\begin{align}
\frac{1}{t}\log\|P_X^{t}(x)|_{N_{x}^1}\|\geq\delta.\nonumber
\end{align}
Since $x\in\Sigma(X)$, there are $X_n\in\mathcal{U}$, $C^1$-converging to $X$, and $p_n\in Per(X_n)$ with period $\pi_n$. Notice that $\pi_n\rightarrow+\infty$ as $n\rightarrow\infty$, otherwise, $x\in Per(X)$ with period $\pi$ such that $P_X^{\pi}(x)|_{N_{x}^1}$ expands, by (\ref{expansion}). This is a contradiction since $X\in\mathcal{G}^1_{\mu}(M)$.
So, assuming that $\pi_n>t_{\delta}$ for every $n$, by the continuity of the dominated splitting we have that, for $n$ big enough, $$\|P_{X_n}^{\pi_n}(p_n)|_{N_{p_n}^1}\|\geq\exp(\delta\pi_n)>\theta^{\pi_n}.$$
But this contradicts (a) in Lemma~\ref{unifhyper}, because $X_n\in \mathcal{U}$. So, $P_X^t|_{N^1}$ is uniformly contracting on $M$.

Analogously, we prove that $P^t_X|_{N^2}$ is uniformly expanding on $M$, using $(b)$ of Lemma \ref{unifhyper}. Thus, $M$ is Anosov.
\begin{flushright}
$\square$
\end{flushright}

We end this section with the proof of Corollary \ref{maincor}.

\begin{proof}[Proof of Corollary~\ref{maincor}] 
By contradiction, assume that there exists an isolated vector field $X$ on the boundary of $\mathcal{A}_{\mu}(M^n)$, for fixed $n\geq 4$. 
In this case, we claim that $Sing(X)=\emptyset$. 
Let us suppose that this claim is not true.
If $p\in Sing(X)$ is hyperbolic, and so persistent to small $C^1$-perturbations of $X$, we can find a divergence-free vector field $Y$, arbitrarily close to $X$, such that $Sing(Y)\neq\emptyset$. 
But this is a contradiction because, since $X$ is isolated on the boundary of $\mathcal{A}_{\mu}(M^n)$, $Y$ has to be Anosov. 
If $p$ is not hyperbolic, by Lemma \ref{linearsing}, we can transform $p$ in a hyperbolic singularity of a vector field $Z$, that is $C^1$-close to $X$. So, as before, we reach a contradiction.

Now, by the previous claim and by Remark \ref{remlemma}, we deduce that $M$ is regular for $X$ and that $P^t_X$ admits a dominated splitting over $M$. So, we just have to follow the proof of Theorem \ref{mainth}, presented in Section \ref{th1}, in order to conclude that $X\in\mathcal{A}_{\mu}(M^n)$, which is a contradiction. So, the boundary of $\mathcal{A}_{\mu}(M^n)$ cannot have isolated points.
\end{proof}

\end{section}

\begin{section}{Proof of Theorem \ref{structh}}\label{strstasec}

Let $X\in\mathfrak{X}^1_{\mu}(M)$ be a $C^1$-structurally stable vector field, where $M$ is a manifold with dimension $n\geq4$, and choose a $C^1$-neighborhood $\mathcal{V}$ of $X$, such that every $Y\in\mathcal{V}$ is topologically conjugated to $X$. 

We start the proof with the following claim.

\begin{claim}\label{strucclaim}
If $X\in\mathfrak{X}^1_{\mu}(M)$ is $C^1$-structurally stable then $Sing(X)=\emptyset$.
\end{claim}

\begin{proof}[Proof of Claim \ref{strucclaim}]
Let $X\in\mathfrak{X}^1_{\mu}(M)$ be a $C^1$-structurally stable vector field and suppose that there exists $p\in Sing(X)$. 
If $p$ is a linear hyperbolic saddle then, perturbing $X$ in $\mathcal{V}$ and proceeding as in the proof of Lemma~\ref{Gnosing}, we get a contradiction. This happens because the existence of a $C^1$-perturbation $Y$ of $X$ in $\mathcal{V}$, such that $P^{\pi}_{Y}(x)=id$, is forbidden, for $x\in Per(Y)$ with arbitrarily large period $\pi$, because the existence of a parabolic closed orbit prevents the structural stability (see \cite{R}).
Now, if $p\in Sing(X)$ is not hyperbolic then, as stated in Lemma~\ref{linearsing}, it can be turned into a linear hyperbolic singularity by a small $C^1$-perturbation of $X$ in $\mathcal{V}$. Then, we just have to apply the previous argument once again.
So, a $C^1$-structurally stable divergence-free vector field has no singularities, which concludes the proof of Claim~\ref{strucclaim}. 
\end{proof}

Now, we are in conditions to obtain the conclusion of the proof of Theorem \ref{structh}. 
Recall that $\mathcal{PR}^1_{\mu}(M)$ is a residual set in $\mathfrak{X}^1_{\mu}(M)$, such that any $X\in\mathcal{PR}^1_{\mu}(M)$ has its closed orbits dense in $M$. Let $\mathcal{W}$ be a small $C^1$-neighbourhood of $X$ such that Theorem~\ref{BGV} holds. 
We start $C^1$-perturbing $X$ in $\mathcal{V}\cap\mathcal{W}$ and taking a vector field $Y\in\mathcal{V}\cap \mathcal{W}\cap\mathcal{PR}^1_{\mu}(M)$. So, $Y$ is a structurally stable divergence-free vector field and it has a closed orbit $y$ with arbitrarily large period $\pi_y$. 
So, by Theorem~\ref{BGV}, there is $\ell>0$ such that $P^t_Y$ admits an $\ell$-dominated splitting over the $Y^t$-orbit of $y$ because, since $Y$ is structurally stable, the existence of a parabolic orbit associated to $\tilde{Y}\in\mathcal{V}$ is not allowed. Then, reproducing the technique used in the proof of Lemma~\ref{Gnosing}, we conclude that $P^t_Y$ admits a dominated splitting over $M$ since, by Claim \ref{strucclaim}, $Sing(Y)=\emptyset$.
Finally, following the proof of Theorem \ref{mainth}, we show that $Y\in\mathcal{A}_{\mu}(M^n)$. 
\begin{flushright}
$\square$
\end{flushright}

\begin{remark}
We do not show that the vector field $X$ in the previous result is itself Anosov, which would prove the stability conjecture for higher dimensional divergence-free vector fields (see example in \cite{Go}).
\end{remark}

\end{section}


\begin{section}{Proof of Theorem \ref{mainth2}}\label{th2}

Let us start by proving the following lemmas. Recall that $\mathcal{FC}^1_{\mu}(M)$ is the set of divergence-free vector fields that are far from heterodimensional cycles and that $\mathcal{KS}^1_{\mu}(M)$ denotes the Kupka-Smale $C^1$-residual set in $\mathfrak{X}^1_{\mu}(M)$.

\begin{lemma}\label{constantind}
There exists a residual set $\mathcal{S}\subset\mathcal{FC}^1_{\mu}(M)$ such that, for every $X\in\mathcal{S}$, all the critical elements of $X$ are hyperbolic and their index is constant.
\end{lemma}

\begin{proof}
Take $\mathcal{S}=\mathcal{FC}^1_{\mu}(M)\cap\mathcal{KS}^1_{\mu}(M)$, a residual set in $\mathcal{FC}^1_{\mu}(M)$, and assume that $X\in\mathcal{S}$ admits two critical elements $\Delta_X$ and $\Gamma_X$ with different indices, say $ind(\Delta_X)<ind(\Gamma_X)$. 
Notice that $\Delta_X$ and $\Gamma_X$ can be closed orbits or singularities.
Let $\mathcal{U}$ be an arbitrarily small $C^1$-neighborhood of $X$, in $\mathfrak{X}^1_{\mu}(M)$, such that the hyperbolic continuations of $\Delta_X$ and of $\Gamma_X$ are well defined.

By Theorem \ref{topmix}, there exists $Y\in\mathcal{U}\cap\mathcal{S}$ which is topologically mixing, so transitive. Let $\Delta'_Y\in\mathcal{O}_Y(\Delta_Y)$ and $\Gamma'_Y\in\mathcal{O}_Y(\Gamma_Y)$. So, if we take $p\in W_Y^s(\Delta'_Y)$ and $q\in W_Y^u(\Gamma'_Y)$, there exists an orbit of $Y$ which passes arbitrarily close to $p$ and $q$. Now, applying the conservative version of the Connecting Lemma for flows (see \cite{WX}), we get $Z\in\mathcal{U}$ such that $W_Z^s(\Delta'_Z)$ and $W_Z^u(\Gamma'_Z)$ intersects transversely. Finally, we can repeat the previous argument, in order to get $\tilde{Z}$, $C^1$-arbitrarily close to $Z$, such that $W_{\tilde{Z}}^u(\Delta'_{\tilde{Z}})\cap W_{\tilde{Z}}^s(\Gamma'_{\tilde{Z}})\neq \emptyset$, but also $W_{\tilde{Z}}^s(\Delta'_{\tilde{Z}})\cap W_{\tilde{Z}}^u(\Gamma'_{\tilde{Z}})\neq ~\emptyset$, because the  first connection is robust and so it persists to small perturbations. Thus, $\tilde{Z}$ exhibits a heterodimensional cycle in $\mathcal{S}$, that can be a periodic, a singular or a mixed heterodimensional cycle. So, we reach a contradiction. Then every critical element of any $X$ in $\mathcal{S}$ is hyperbolic and has constant index.
\end{proof}


\begin{lemma}\label{FCG}
If $X\in  \mathfrak{X}^1_{\mu}(M)\backslash  \overline{\mathcal{G}^1_{\mu}(M)}$ then $X$ can be $C^1$-approximated by a divergence-free vector field  $\;Y$ that exhibits a heterodimensional cycle. Notice that, in particular, $\mathcal{FC}^1_{\mu}(M)\subset\overline{\mathcal{G}^1_{\mu}(M)}$.
\end{lemma}

\begin{proof}
Take $X\in \mathfrak{X}^1_{\mu}(M)\backslash \overline{\mathcal{G}^1_{\mu}(M)}$ and $Y$ a vector field, $C^1$-sufficiently close to $X$, such that $Y\in\big(\mathfrak{X}^1_{\mu}(M)\backslash\overline{\mathcal{G}^1_{\mu}(M)}\big)\cap\mathcal{KS}^1_{\mu}(M)\cap\mathcal{PR}^1_{\mu}(M)$, where $\mathcal{KS}^1_{\mu}(M)$ denotes the Kupka-Smale $C^1$-residual subset of $\mathfrak{X}_{\mu}^1(M)$ (see \cite{S}) and $\mathcal{PR}^1_{\mu}(M)$ denotes a $C^1$-residual set in $\mathfrak{X}^1_{\mu}(M)$ such that any $X\in\mathcal{PR}^1_{\mu}(M)$ has its closed orbits dense in $M$. 
So, we can take $p_Y$ a hyperbolic closed orbit of $Y$, with period $\pi_Y$ and index $u$. 
Let $\mathcal{W}$ be a small $C^1$-neighborhood of $Y$ such that the hyperbolic continuation of $p_Y$ is well defined.

As $Y\in \mathfrak{X}^1_{\mu}(M)\backslash \overline{\mathcal{G}^1_{\mu}(M)}$, and once that this set is open, for every $C^1$-neighborhood $\mathcal{V}$ of $Y$ in $\mathfrak{X}^1_{\mu}(M)\backslash \overline{\mathcal{G}^1_{\mu}(M)}$, we can find a vector field $Z\subset\mathcal{W}\cap\mathcal{V}$, $C^1$-arbitrarily close to $Y$, such that $Z$ has a hyperbolic closed orbit $p_Z$, corresponding to the hyperbolic continuation of $p_Y$, with index $u$ and period $\pi_Z$ close to $\pi_Y$.
However, since \linebreak$Z\in\mathfrak{X}^1_{\mu}(M)\backslash \overline{\mathcal{G}^1_{\mu}(M)}$, it has to have a non-hyperbolic critical point $q_Z$, that can be a singularity or a closed orbit. 

If $q_Z$ is a non-hyperbolic singularity of $Z$, by a $C^1$-small perturbation on the vector field, it can be turned on a hyperbolic singularity with index $v\neq u$. 
Observe that it can appear another non-hyperbolic critical points but, since we already have two hyperbolic critical points with different indices, we can build a heterodimensional cycle, as shown in Lemma \ref{constantind}.

Now, assume that $q_Z$ is a non-hyperbolic closed orbit of $Z$, so unstable. In this case, we start by applying Theorem \ref{zuppa} to increase the differentiability of the vector field $Z$ from $C^1$ to $C^4$, in order to be able to use Theorem~\ref{Franks}, that ensures the existence of a vector field $W\in\mathfrak{X}^4_{\mu}(M)$, $C^1$-close to $Y$, such that $p_W$ (the hyperbolic continuation of $p_Z$) and $q_W$ are now hyperbolic closed orbits for $W$ with different indices. Again, by Lemma \ref{constantind}, we can $C^1$-proximate $W$ by a vector field $\tilde{W}$ exhibiting a heterodimensional cycle. 
\end{proof}

\begin{remark}
For the dissipative setting, in \cite{GW} the authors show that $\mathcal{G}^1(M)$ is a subset of $\mathcal{FC}^1(M)$.
\end{remark}

Now, we are ready to prove Theorem \ref{mainth2}. Suppose that $X\in~\mathfrak{X}^1_{\mu}(M)$, where $\dim(M)\geq 4$, cannot be $C^1$-approximated by a divergence-free vector field exhibiting a heterodimensional cycle, meaning that $X\in \mathcal{FC}^1_{\mu}(M)$. Then, by Lemma \ref{FCG}, and due to the openness of $\mathcal{FC}^1_{\mu}(M)$ in $\mathfrak{X}^1_{\mu}(M)$, we see that $X$ can be $C^1$-approximated by a di\-ver\-gence-free vector field $Y$, satisfying that $Y\in\mathcal{FC}^1_{\mu}(M)\cap\mathcal{G}^1_{\mu}(M)$. Finally, Theorem~\ref{mainth} ensures that $Y$ is Anosov, which concludes the proof.
\begin{flushright}
$\square$
\end{flushright}

\end{section}

\section*{Acknowledgements}
I would like to thank my supervisors, M\'ario Bessa and Jorge Rocha, whose encouragement and guidance enabled me to develop this work. 

The author was supported by Funda\c c\~ao para a Ci\^encia e a Tecnologia, SFRH/BD/
33100/2007.



\begin{thebibliography}{ABC}
\bibitem{AP} Andronov, A., Pontrjagin, L., \emph{Syst\`emes grossiers.} Dokl. Akad. Nauk. SSSR, 14 (1937), 247--251.
\bibitem{AM} Arbieto, A., Matheus, C., \emph{A pasting lemma and some applications for conservative systems.} Ergodic Theory Dynam. Systems, 27 (2007), 1399--1417.
\bibitem{Ar} Arnaud, M-C.,  \emph{Le "`closing lemma"' en topologie $C^1$.} Mem. Soc. Math. Fr., Nouv. S\'{e}rie 74 (1998), 1--120. 
\bibitem{AH} Arroyo, A., Hertz, R., \emph{Homoclinic bifurcation and uniform hyperbolicity for three-dimensional flows.} Ann. I. H. Poincar\'{e} AN, 20 (5) (2003), 805--841.
\bibitem{B1} Bessa, M., \emph{Generic incompressible flows are topological mixing.} Comptes Rendus Mathematique, 346 (2008), 1169--1174.  
\bibitem{BFR} Bessa, M., Ferreira, C., Rocha, J., \emph{On the stability of the set of hyperbolic closed orbits of a Hamiltonian.} To appear in Math. Proc. Camb. Phil. Soc.
\bibitem{BR0} Bessa, M., Rocha, J., \emph{Contributions to the geometric and ergodic theory of conservative flows.} Preprint (2008), arXiv:0810.3855.
\bibitem{BR1}Bessa, M., Rocha, J., \emph{Homoclinic tangencies versus uniform hyperbolicity for conservative 3--flows.} Journal of Differential Equations, 247 (2009), 2913--2923.
\bibitem{BR} Bessa, M., Rocha, J., \emph{On $C^1$-robust transitivity of volume-preserving flows.} Journal of Differential Equations, 245 (11) (2008), 3127--3143.
\bibitem{BR2} Bessa, M., Rocha, J., \emph{Three-dimensional conservative star flows are Anosov.} Discrete and Continuous Dynamical Systems-A, 26 (3) (2010), 839--846.
\bibitem{BR3} Bessa, M., Rocha, J., \emph{Anosov versus Heterodimensional cycles: A $C^1$ dichotomy for conservative maps.} Preprint (2009).\\ http://cmup.fc.up.pt/cmup/bessa/cycles.pdf
\bibitem{Di} Ding, H., \emph{Disturbance of the homoclinic trajectory and applications.} Acta Sci. Nat. Univ. Pekin., 1 (1986), 53--63.
\bibitem{F} Franks, J., \emph{Necessary conditions for stability of diffeomorphisms.} Trans. A.M.S., 158 (1971), 301--308.
\bibitem{Ga} Gan, S., \emph{The star systems $\mathfrak{X}^*$ and a proof of the $C^1$ $\Omega$-stability conjecture for flows.} Journal of Differential Equations, 163 (1) (2000), 1--17.
\bibitem{GW}  Gan, S., Wen, L., \emph{Nonsingular star flows satisfy Axiom A and the no-cycle condition.} Invent. Math, 164 (2006), 279--315.
\bibitem{Go} Gogolev, A., \emph{Diffeomorphisms Holder conjugate to Anosov diffeomorphisms.} Ergodic Theory Dynam. Systems, 30 (2010), 441--456.
\bibitem{G} Guchenheimer, J., \emph{A strange, strange attractor.} The Hopf bifurcation theorems and its applications. Applied Mathematical Series, 19, 368--381. Springer-Verlag (1976).
\bibitem{Gu} Gutierrez, C., \emph{A counter-example to a $C^2$ closing lemma.} Ergodic Theory Dynam. Systems, 7 (1987), 509--530.
\bibitem{H} Hayashi, S., \emph{Diffeomorphisms in $\mathcal{F}^1(M)$ satisfy Axiom A.} Ergodic Theory Dynam. Systems,  12 (2) (1992), 233--253.
\bibitem{HPS} Hirsch, M., Pugh, C., Shub, M., \emph{Invariant Manifolds, Lectures Notes in Mathematics 583.} Springer, Berlin (1977).
\bibitem{LW} Li, C., Wen, L., \emph{$\chi^*$ plus Axiom A does not imply no-cycle.} Journal of Differential Equations, 119 (1995).
\bibitem{LGW}Li, M., Gan, S., Wen, L., \emph{Robustly transitive singular sets via approach of an extended linear Poincar\'{e} flow.} Discrete and Continuous Dynamical Systems, 13 (2) (2005), 239--269.
\bibitem{Ma} Ma\~n\'e, R., \emph{A proof of the $C\sp 1$ stability conjecture.}  Inst. Hautes \'Etudes Sci. Publ. Math., 66 (1988), 161--210.
\bibitem{M} Ma\~{n}\'{e}, R., \emph{An ergodic closing lemma.} Ann. of Math., 116 (1982), 503--540.
\bibitem{M1} Ma\~{n}\'{e}, R., \emph{Contributions to the stability conjecture.} Topology, 17 (1978), 383--396.
\bibitem{MP} Morales, C. A., Pacifico, M. J., \emph{Sufficient conditions for robustness of attractors.} Pacific journal of mathematics, 216 (2) (2004), 327--342.
\bibitem{O} Oseledets, V., \emph{A multiplicative ergodic theorem: Lyapunov characteristic numbers for dynamical systems.} Trans. Moscow Math. Soc., 19 (1968), 197--231.
\bibitem{P} Palis, J., \emph{A global view of dynamics and a conjecture on the denseness of finitude of attractors.} Asterisque, 261 (2000), 339--351.
\bibitem{PM} Palis, J., de Melo, W., \emph{Geometric Theory Of Dynamical Systems.} Springer-Verlag (1982).
\bibitem{PT} Palis, J., Takens, F., \emph{Hyperbolicity and Sensitive Chaotic Dynamics of Homoclinic Bifurcations.} Cambridge Univ. Press, Cambridge (1993).
\bibitem{Pu} Pugh, C., \emph{Against the $C^2$ closing lemma.} Journal of Differential Equations, 17 (1975), 435--43.
\bibitem{PR} Pugh, C., Robinson, C., \emph{The $C^1$ closing lemma, including Hamiltonians.} Ergodic Theory Dynam. Systems, 3 (1983), 261--313. 
\bibitem{PS} Pujals, E., Sambarino, M., \emph{Homoclinic tangencies and hyperbolicity for surface diffeomorphisms.} Ann. of Math., (2) 151 (3) (2000), 961--1023.
\bibitem{PS2} Pujals, E., Sambarino, M., \emph{On the dynamic of dominated splitting}. Ann. of Math., (2) 169 (3) (2009), 675--739. 
\bibitem{R0} Robinson, C., \emph{Generic properties of conservative systems.} Amer. J. Math., 92 (1970), 562--603.
\bibitem{R} Robinson, C., \emph{Lectures on Hamiltonian Systems.} Monografias de Ma\-te\-m\'a\-ti\-ca IMPA (1971).
\bibitem{S} Smale, S., \emph{Differentiable dynamical systems.} Bull. Amer. Math. Soc. 73 (1967), 747--817.
\bibitem {SS} Sternberg, S., \emph{On the behavior of invariant curves near a hyperbolic point of a surface transformation.} Amer. J. Math., 77 (1955), 526–-534.
\bibitem{To} Toyoshiba, H., \emph{Vector fields in the interior of Kupka-Smale systems satisfy axiom A.} Journal of Differential Equations, 177 (2001), 27--48.
\bibitem{V1} Vivier, T., \emph{Projective hyperbolicity and fixed points.} Ergodic Theory Dynam. Systems, 26 (2006), 923--936.
\bibitem{W} Wen, L., \emph{On the $C^1$-stability conjecture for flows.} Journal of Differential Equations, 129 (1996), 334--357.
\bibitem{WX} Wen, L., Xia, Z., \emph{$C^1$ connecting lemmas.} Trans. Amer. Math. Soc., 352 (2000), 5213--5230.
\bibitem{Z} Zuppa, C., \emph{Regularisation $C^{\infty}$ des champs vectoriels qui pr\'{e}\-ser\-vent l'e\-l\'{e}\-ment de vo\-lu\-me.} Bol. Soc. Bras. Mat., 10 (2) (1979), 51--56.
\end{thebibliography}
\end{document}